\documentclass[11pt]{amsart}
\usepackage{enumerate}
\usepackage{amsopn,amsfonts,amsmath,amssymb}
\usepackage[dvips]{graphicx}
\usepackage{float}
\usepackage{subfigure}
\usepackage{dcolumn}
\usepackage{color, colortbl}
\usepackage{booktabs}
\usepackage[usenames,dvipsnames,svgnames,table]{xcolor}
\usepackage[colorlinks=true,
            linkcolor=red,
            urlcolor=blue,
            citecolor=gray]{hyperref}
\allowdisplaybreaks
\usepackage{geometry}  
\def\indic{{\rm {\large 1}\hspace{-2.3pt}{\large
l}}}

\theoremstyle{definition}
\newtheorem{theorem}{Theorem}

\newtheorem{rmrk}{Remark}
\newtheorem{algorithm}{Algorithm}

\usepackage[foot]{amsaddr}

\begin{document}
\title[]{Relaxing Monotonicity in Endogenous Selection Models and Application to Surveys}

\author[Gautier]{Eric Gautier$^{(1)}$}
\email{\href{mailto:eric.gautier@tse-fr.eu}{eric.gautier@tse-fr.eu}}
\date{This version: \today.} 

\thanks{The authors acknowledge financial support from the grants ERC POEMH 337665 and ANR-17-EURE-0010. The material of this paper was presented at the $7^{\mathrm{eme}}$ Colloque Francophone sur les Sondages at ENSAI in 2012.}

\begin{abstract}
This paper considers endogenous selection models, in particular nonparametric ones. 
Estimating the unconditional law of the 
outcomes 
is possible when one uses instrumental variables. Using a selection equation which is additively separable in a one dimensional unobservable has the sometimes undesirable property of instrument monotonicity. We present models which allow for nonmonotonicity and are based on nonparametric random coefficients indices. We discuss their nonparametric identification and apply these results to inference on nonlinear statistics such as the Gini index in surveys when the nonresponse is not missing at random.
\end{abstract}
\maketitle

\section{Introduction}
Empirical researchers often face a missing data problem. This is also called selection. Due to missing data, the observed data on an outcome variable corresponds to draws from the law of the outcome conditional on nonmissingness. Most of the time, the law of interest is the unconditional one. But the researcher can also be interested in the law of the outcome variable for the population that does not reveal the value of the outcome. For example surveys rely on a sample drawn at random and the estimators require the observation of all sampled units. In practice there is missing data and those estimators cannot be computed. A common practice is to rely on imputations. This means that the missing outcomes are replaced by artificial ones so that the estimator can eventually be computed. In the presence of endogenous selection, the law conditional on nonselection is the important one for imputation. 

It is usual to assume that the data is Missing at Random (henceforth MAR, see \cite{LR}) in which case there are perfectly observed variables such that the law of the outcome conditional on them and selection is the same as the law of outcome conditional on them and nonselection. Under such an assumption, the estimable conditional law is the same as the one which is unconditional on selection. As a consequence the researcher does not need a model for the joint law of the outcome and selection and the selection can be ignored. In survey sampling, the sampling frame can be based on variables available for the whole population for example if it involves stratification. In this case those variables are natural candidates for conditioning variables for MAR to hold.  
In practice there is noncompliance. It means that the researcher often does not have observations for all sampled units. This is called missing data in survey statistics. Though the original sampling law is known, the additional layer of missing data can be viewed as an additional selection mechanism conditional on the first one. The law of this second selection mechanism is unknown to the statistician. Oftentimes it can be suspected that units reveal the value of a variable partly depending on the value of that variable and the MAR assumption does not hold. This is a type of endogeneity issue commonly studied in econometrics. For example, wages are only observed for those who work and those who do not work might prefer not to work becuase their wage would be too low. Firms only carry out investment decisions if the net discounted value is nonnegative. An individual might be less willing to answer a question on his salary because it is not a typical one (either low or high). We expect a strong heterogeneity in the mechanism that drives individuals to not reveal the value of a variable.   

When the MAR assumption no longer holds, the selection mechanism cannot be ignored. Identification of the  distribution unconditional on selection or the distribution conditional on nonselection usually relies on the specification of a model for the vector formed by the outcome and a binary variable for selection. The alternative approach is to follow the partial identification route and recognize that the parameters of interest which are functionals of these distributions lie in sets. 
 The Tobit and generalized Tobit models (also called Heckman selection model, see \cite{H}) are classical parametric selection models to handle endogenous selection. The generalized Tobit model involves a system of two equations: one for the outcome and one for the selection. Each of these equations involve an  error term and these errors are dependent, hence endogeneity. Identification in such systems relies on some variables which appear in the selection equation and are not measurable with respect to the sigma-field generated by the variables in the outcome equation and which do not have an effect on the errors. So these variables have an effect on the selection but not on the outcome. They are called instrumental variables or instruments. 
 
This paper presents nonparametric models in sections \ref{s2} and \ref{s3}. We explain in Section \ref{s3} that having a one dimensional error term appearing in an additively separable form in the selection equation implies so-called instrument monotonicity. Instrument monotonicity has been introduced in \cite{IA}. It has a strong identification power but at the same time leads to unrealistic selection equations as we detail in Section \ref{s3}. To overcome this issue, we present in Section \ref{RC} selection equations where the error in the selection equation is multidimensional and appears in a non additively separable fashion. The baseline specification is a model where the selection equation involves an index with random coefficients. We show that we can rely on a nonparametric models for these random coefficients. Finally, Section \ref{sec:1} presents a method to obtain a confidence interval around a nonlinear statistic like the Gini index with survey data in the presence of non MAR missing data when we suspect that some instruments are nonmonotonic. These confidence intervals account for both the uncertainty due to survey sampling and the one due to missing data.

\section{Preliminaries}\label{s1}
\subsection{Notations}
Bold letters are used for vectors and matrices and capital letters for random elements. In the presence of an identically distributed sample we add an index $i$ for the marginal random element of index $i$. $1\{\cdot\}$ denotes the indicator function, $\partial_p$ the derivative with respect to the variable $p$, $\langle\cdot,\star\rangle$ the inner product in the Euclidian space, $\|\cdot\|$ the euclidian norm, $\sigma$ the spherical measure on the unit sphere 
in the Euclidian space. We denote it by $\mathbb{S}^{d-1}$ when the Euclidian space is $\mathbb{R}^d$. $|\mathbb{S}^{d-1}|$ is its area.  We write a.e. for almost everywhere. 
$\mathbb{N}_0$ is the set of nonnegative integers and $\mathbb{N}$ are the positive integers. Quasi-analytic functions are functions which are infinitely differentiable and are characterized by the value of a function and all its derivatives at a point. A quasi-analytic class is defined via certain controls on the sup-norm of all the derivatives (Laplacians for functions defined on the sphere) as explained for example in \cite{GG}. Analytic functions are quasi-analytic.  

All random elements are defined on the same probability space with probability $\mathbb{P}$ and $\mathbb{E}$ is the expectation. The support of a function or random vector is denoted by $\mathrm{supp}$. We denote by $\mathrm{supp}(U|\boldsymbol{X}=\boldsymbol{x})$ the support of the conditional law of $U$ given $\boldsymbol{X}=\boldsymbol{x}$ when it makes sense. For a random vector $\boldsymbol{\boldsymbol{\Gamma}}$, $f_{\boldsymbol{\boldsymbol{\Gamma}}}$ is its density with respect to a measure which will be clear in the text and $d_{\boldsymbol{\boldsymbol{\Gamma}}}$ is its dimension. We use the notation $f_{\boldsymbol{\boldsymbol{\Gamma}}|\boldsymbol{X}=\boldsymbol{x}}$ for a conditional density and $\mathbb{E}[U|\boldsymbol{X}=\boldsymbol{x}]$ for the conditional expectation function evaluated at $\boldsymbol{x}\in\mathrm{supp}(\boldsymbol{X})$. Below we usually write for all $\boldsymbol{x}\in\mathrm{supp}(\boldsymbol{X})$ as if $\boldsymbol{X}$ were discrete. If $\boldsymbol{X}$ 
is continuous it should often be replaced by a.e.. If $\boldsymbol{X}$ has both a discrete and a continuous component then the ``for all" statement should hold for $\boldsymbol{x}$ in the part of the support which is discrete. Equalities between random variables are understood almost surely. Random vectors appearing in models and which realisations are not in the observed data are called unobservable.

\subsection{Baseline setup}
In this paper, the researcher is interested in features of the law of a variable $Y$ given $\boldsymbol{X}=\boldsymbol{x}$, where $\boldsymbol{x}\in\text{supp}(\boldsymbol{X})$. She has a selected sample of observations of $Y$, observations of a vector of which $\boldsymbol{X}$ is a subvector,  for the selected and unselected samples, and $R$ is a binary variable equal to 1 when $Y$ is observed and  else is 0.  In this paper the selection is often interpreted as a response and nonselection as nonresponse. 

The law of $Y$ given $\boldsymbol{X}=\boldsymbol{x}$ is (nonparametrically) identified if, for a large class $\Phi$ of measurable functions $\phi$, $\mathbb{E}[\phi(Y)|\boldsymbol{X}=\boldsymbol{x}]$ can be characterized from the model equation, the  restrictions on the primitives (such as conditional independence when using an instrumental variables strategy), and the distribution of the observed data. It is possible to take for $\Phi$ the bounded measurable functions, the bounded and continuous functions, the set of indicator functions $1\{\cdot \le t\}$ for all $t\in\mathbb{R}$, the set of functions $\cos(t\ \cdot)$ and $\sin(t\ \cdot)$ for all $t\in\mathbb{R}$ (or certain countable subsets if $Y$ is bounded). 
It is possible to add the assumption that $\Phi$ only contains functions which are nonnegative (a.e. if $Y$ is continuous). For example, one can work with the functions $\cos(t\ \cdot)+1$, and $\sin(t\ \cdot)+1$ for all $t\in\mathbb{R}$. We call such class $\Phi$ an identifying class. 
 
One can deduce, from the law of $Y$ given $\boldsymbol{X}=\boldsymbol{x}$, the law of a variable $Y$ given $\boldsymbol{X}=\boldsymbol{x}$, where $\boldsymbol{x}\in\text{supp}(\boldsymbol{X})$, and $R=0$. This is the law of the outcome for the nonrespondants. It is the useful one for imputation. For all $\boldsymbol{x}\in\mathrm{supp}(\boldsymbol{X})$,  we have
\begin{equation}\label{e6b}
\mathbb{E}[\phi(Y)|\boldsymbol{X}=\boldsymbol{x},R=0]=\frac{\mathbb{E}[\phi(Y)|\boldsymbol{X}=\boldsymbol{x}]-\mathbb{E}[\phi(Y)R|\boldsymbol{X}=\boldsymbol{x}]}{\mathbb{P}(R=0|\boldsymbol{X}=\boldsymbol{x})}.
\end{equation}

This paper presents identification results for a more fundamental object which is the joint distribution of the outcome and unobservable in the selection equation given $\boldsymbol{X}=\boldsymbol{x}$, where $\boldsymbol{x}\in\text{supp}(\boldsymbol{X})$. One can clearly deduce from it by marginalization the distribution of $Y$ given $\boldsymbol{X}=\boldsymbol{x}$, where $\boldsymbol{x}\in\text{supp}(\boldsymbol{X})$. 

\subsection{NMAR missing data}
Let $\boldsymbol{W}$ be a vector, of which $\boldsymbol{X}$ is a subvector, which is observed for the selected and unselected samples.  
Inference on the conditional law of $Y$ given $\boldsymbol{X}$ is possible if $Y$ and $R$ are independent given $\boldsymbol{W}$, namely if, for all bounded continuous function $\phi$,
\begin{equation}\label{e1}\mathbb{E}[\phi(Y)R|\boldsymbol{W}]=\mathbb{E}[\phi(Y)|\boldsymbol{W}]\mathbb{E}[R|\boldsymbol{W}]
\end{equation}
in which case
\begin{equation}\label{e1b}\mathbb{E}[\phi(Y)|\boldsymbol{W}]=\mathbb{E}[\phi(Y)|\boldsymbol{W},R=1]
\end{equation}
and we conclude by the law of iterated expectations. 
Condition \eqref{e1} is called Missing at Random (MAR, see \cite{LR}). When it holds without the conditioning on $\boldsymbol{W}$, it is called Missing Completely at Random (MCAR). In econometrics $\boldsymbol{W}$ such that \eqref{e1b} holds is called a control variable. 

We consider cases where the researcher does not know that a specific vector $\boldsymbol{W}$ is such that \eqref{e1} holds. Then $R$ is partly based on $Y$, even conditionally.  This situation is called Not Missing at Random (NMAR, see \cite{LR})\footnote{The terminology nonignorable is also used but is defined for parametric models and requires parameter spaces to be rectangles. This is why we do not use this terminology in this paper.}. In the language of econometrics, this is called endogenous selection.

To handle NMAR missing data it is usual to rely on a joint model for the determination of $Y$ and $R$. This paper considers so-called triangular systems where a model for $R$, called a selection equation, is specified and does not involve $Y$ but the dependence occurs via dependent latent (unobserved) variables. The identification arguments below rely on a vector $\boldsymbol{Z}$ of so-called instrumental variables which is observed for the selected and unselected samples.  It plays a completely different role as $\boldsymbol{W}$ and $\boldsymbol{X}$. 


\section{Models with One Unobservable in the Endogenous Selection}\label{s2}
Important parametric models rely on $Y=\boldsymbol{X}^{\top}\boldsymbol{\beta}+\sigma E_Y$ as a model equation for the variable of interest, $\boldsymbol{\beta}$ and $\sigma$ are unknown parameters, $\boldsymbol{X}$ and $E_Y$ are independent, and $E_Y$ is a standard normal random variable.
In the Tobit model, $R=1\{Y>y_L\}$ for a given threshold $y_L$. In the Heckman selection model (see \cite{H}) 
\begin{align}
&R=1\{\boldsymbol{Z}^{\top}\boldsymbol{\boldsymbol{\gamma}}- E_R>0\},\label{eH2}\\
&(E_Y,E_R)\text{ and }(\boldsymbol{X}^{\top},\boldsymbol{Z}^{\top})\text{ are independent},\notag\\
&(E_Y,E_R)^{\top}\text{ is a mean zero gaussian vector with covariance matrix 
}\left(
\begin{array}{ccc}
  1 & \rho  \\
 \rho & 1  
\end{array}
\right).\notag
\end{align}
\eqref{eH2} is the selection equation. 
The law of $Y$ given $\boldsymbol{X}$ and $\boldsymbol{Z}$, hence of $Y$ given $\boldsymbol{X}$ is identified and the model parameters can be estimated by maximum likelihood. Some functionals of the conditional law of $Y$ given $\boldsymbol{X}$ can be estimated for some semi-parametric extensions. For example, the conditional mean function can be obtained by estimating a regression model with an additional regressor which is a function of $\boldsymbol{Z}^{\top}\gamma$. This leads to the interpretation that the endogeneity can be understood as a missing regressor problem. 

A more general model is
\begin{align}
&R=1\{\pi(\boldsymbol{Z})>H\},\label{e2}\\
&\boldsymbol{Z}\text{ is independent of }(H,Y)\text{ given }\boldsymbol{X}\label{e3},\\
&\text{For all}\ \boldsymbol{x}\in\mathrm{supp}(\boldsymbol{X}),\ \text{the law of }H\text{ given } \boldsymbol{X}=\boldsymbol{x}\text{ is uniform on }(0,1),\label{e4}\\
&\text{For all}\ \boldsymbol{x}\in\mathrm{supp}(\boldsymbol{X}),\ \mathrm{supp}\left(\pi(\boldsymbol{Z})|\boldsymbol{X}=\boldsymbol{x}\right)=[0,1].\label{e5} 
\end{align}
Equation \eqref{e2} is the selection equation. This model is quite general and clearly
$\pi(\boldsymbol{Z})=\mathbb{E}[R|\boldsymbol{X},\boldsymbol{Z}]=\mathbb{E}[R|\boldsymbol{Z}]$. 
\eqref{e2} is as general as a selection equation that would be defined as $R=1\{g(\boldsymbol{Z})>E_R\}$, where $g$ and the law of $E_R$ are unknown. 
Indeed, one would obtain \eqref{e2} from it by applying the nondecreasing CDF of $E_R$ on both sides of the inequality 
$g(\boldsymbol{Z})>E_R$.
If we replace \eqref{e3} by $H$ and $Y$ are independent given $\boldsymbol{X},\boldsymbol{Z}$, assumption MAR holds by taking $\boldsymbol{W}$ a vector which components are those of $\boldsymbol{X}$ and $\boldsymbol{Z}$.
Condition \eqref{e3}  allows for dependence between $H$ and $Y$ and $R$ to be partly based on $Y$, even conditionally. The vector $\boldsymbol{Z}$ thus plays a very different role from $\boldsymbol{W}$ in the MAR assumption.  By \eqref{e3}, $\boldsymbol{Z}$ has a direct effect on $R$ via $\pi(\boldsymbol{Z})$ which is non trivial but it does not have an effect on $Y$ given $\boldsymbol{X}$. This type of properties for  $\boldsymbol{Z}$ is what makes it a vector of instrumental variables. It provides an alternative identification strategy. 
\eqref{e5} can be replaced by 
\begin{equation}\label{e5b}
\left.
\begin{array}{ll}
\text{For all}\ \boldsymbol{x}\in\mathrm{supp}(\boldsymbol{X}),\\
\ \ \mathrm{supp}\left(\pi(\boldsymbol{Z})|\boldsymbol{X}=\boldsymbol{x}\right)\text{ contains a nonempty open set and,}  \\
\ \ \text{for all }\phi\in\Phi_{\boldsymbol{x}},\ \mathbb{E}[\phi(Y)1\{F(u)>H\}|\boldsymbol{X}=\boldsymbol{x}]\in\mathcal{C}_{\boldsymbol{x}}
\end{array}
\right\},
\end{equation}
for a well chosen CDF $F$ such that $F(x)=\int_{-\infty}^xf(t)dt$, where $f$ is integrable and nonnegative a.e., and  identifying classes $\Phi_{\boldsymbol{x}}$ and $\mathcal{C}_{\boldsymbol{x}}$ a quasi-analytic classes of functions (see \cite{GG}). 

Note that, by \eqref{e0} below, for \eqref{e5b} to hold it is necessary that $f$ is positive a.e. 
Also, 
for all $\boldsymbol{x}\in\mathrm{supp}(\boldsymbol{X})$, the law of $(Y,H)$ conditional on $\boldsymbol{X}=\boldsymbol{x}$ is identified if
\begin{equation}\label{rootident}
\left.
\begin{array}{l}
\text{For all }\boldsymbol{x}\in\mathrm{supp}(\boldsymbol{X}),\text{ there exists an identifying class }\Phi_{\boldsymbol{x}}\\
\text{and functions }F\text{ and }f\text{ as above such, that for all }\phi\in\Phi_{\boldsymbol{x}},\\ 
\mathbb{E}[\phi(Y)|\boldsymbol{X}=\boldsymbol{x},H=F(\cdot)]f(\cdot)\text{ is identified}
\end{array}
\right\}
\end{equation}
We conclude this section by the following result. 
\begin{theorem}\label{t0}
If \eqref{e2}-\eqref{e4} and either \eqref{e5} or \eqref{e5b} hold, then one has \eqref{rootident}. Moreover, 
for all $\boldsymbol{x}\in\mathrm{supp}(\boldsymbol{X})$,  
\begin{align}
\mathbb{E}[\phi(Y)|\boldsymbol{X}=\boldsymbol{x}]&=\int_{\mathbb{R}}\partial_u\mathbb{E}[\phi(Y)R|\boldsymbol{X}=\boldsymbol{x},\pi(\boldsymbol{Z})=F(u)]du\label{e6}\\
&=\mathbb{E}[\phi(Y)R|\boldsymbol{X}=\boldsymbol{x},\pi(\boldsymbol{Z})=1].\label{e6a}
\end{align}
\end{theorem}
\begin{proof}
Based on \eqref{e3}, for all 
$\phi\in\Phi$,  
 $\boldsymbol{x}\in\mathrm{supp}(\boldsymbol{X})$, and $u$ such that $F(u)\in\mathrm{supp}(\pi(\boldsymbol{Z}))$, 
$$\mathbb{E}[\phi(Y)R|\boldsymbol{X}=\boldsymbol{x},\pi(\boldsymbol{Z})=F(u)]=\mathbb{E}[\phi(Y)1\{F(u)>H\}|\boldsymbol{X}=\boldsymbol{x}]$$
so 
\begin{equation}\label{e0}
\mathbb{E}[\phi(Y)|\boldsymbol{X}=\boldsymbol{x},H=F(u)]f(u)=\partial_u\mathbb{E}[\phi(Y)R|\boldsymbol{X}=\boldsymbol{x},\pi(\boldsymbol{Z})=F(u)].
\end{equation}
The conclusion follows from either \eqref{e5} or \eqref{e5b}. By integration we obtain \eqref{e6}, hence 
$$
\mathbb{E}[\phi(Y)|\boldsymbol{X}=\boldsymbol{x}]
=\mathbb{E}[\phi(Y)R|\boldsymbol{X}=\boldsymbol{x},\pi(\boldsymbol{Z})=1]-\mathbb{E}[\phi(Y)R|\boldsymbol{X}=\boldsymbol{x},\pi(\boldsymbol{Z})=0],
$$
and \eqref{e6a} follows because $\mathbb{E}[\phi(Y)R|\boldsymbol{X}=\boldsymbol{x},\pi(\boldsymbol{Z})=0]=0$. 
\end{proof}

\begin{rmrk} Similar formulas as \eqref{e6} and \eqref{e6a} are given for a binary treatment effect model in \cite{HV} for effects that depend on an average (\emph{i.e.} $\phi(y)=y$ for all $y\in\mathbb{R}$) rather than the whole law as above. 
There the integrand is called the local instrumental variable. 
\end{rmrk}

Condition \eqref{e5} is strong.  First, the support of $\boldsymbol{Z}$ should be infinite so in practice we think that at least a variable in $\boldsymbol{Z}$ is continuous. Second, the variation of $\boldsymbol{Z}$ should be large enough to move the selection probability $\pi(\boldsymbol{Z})$ from 0 to 1. 
This is a "large support" assumption. Using \eqref{e6a} for identification is called ``identification at infinity". Using it to construct an estimator does not make an efficient use of the data because it would make use of the subsample for which $\pi(\boldsymbol{Z}_i)$ is close to 1. In contrast, \eqref{e6} can be used to form estimators which use all the data.

The techniques in \cite{GG} allow for supports which are only countable in \eqref{e5}. It is possible to use the techniques in \cite{GG2,GG3} if we replace quasi-analytic by certain analytic classes. The advantage is to be able to use stable Fourier methods for extrapolation to build an estimator.  
\eqref{e5} were not required in the parametric Tobit and Heckman selection models. Condition \eqref{e5b} is a nonparametric middle ground between a parametric assumption made for convenience and a nonparametric one which is often too demanding for finding an instrument. Clearly, in this setup, building an estimator from \eqref{e6a} is simply impossible while building an estimator using \eqref{e6} and the available data is possible.


\section{Monotonicity}\label{s3}
In this section, we show that the above nonparametric specification is not as general as we would think. 
From a modelling perspective, it is related (equivalent, see \cite{V}) to the so-called instrument monotonicity introduced in \cite{IA}.   

For the sake of exposition, assume that $\boldsymbol{Z}$ is discrete. For $\bf{z}\in\mathrm{supp}(\boldsymbol{Z})$ and individuals 
 that we index by $i\in\mathcal{I}(\bf{z})$, 
such that $\boldsymbol{Z}_i=\bf{z}$, we have $R_i=1\{\pi(\mathbf{z})>H_i\}$.  Suppose now that we could change exogeneously (by experimental assignment) $\bf{z}$ to $\bf{z}'$ in $\mathrm{supp}(\boldsymbol{Z})$ leaving unchanged the unobserved characteristics $H_i$ for $i\in\mathcal{I}(\bf{z})$. 
The corresponding $R_i$ of those individuals are shifted monotonically. Indeed, we have either (1) $\pi(\bf{z})\le\pi(\bf{z}')$ or (2) $\pi(\bf{z})>\pi(\bf{z}')$. 
In case (1), 
$$\forall i\in\mathcal{I}(\mathbf{z}),\ 1\{\pi(\mathbf{z})>H_i\}\le 1\{\pi(\mathbf{z}')>H_i\}$$
while in case (2), 
$$\forall i\in\mathcal{I}(\mathbf{z}),\ 1\{\pi(\mathbf{z})>H_i\}\ge 1\{\pi(\mathbf{z}')>H_i\}.$$
This instrument monotonicity condition has been formalized in \cite{IA}.  

Consider a missing data problem in a survey where $d_{\boldsymbol{Z}}=1$, $\boldsymbol{Z}=Z$ is the identity of a pollster, and $R=1$ when the surveyed individual replies and else $R=0$. The identity of the pollster could be Mr A (z=0) or Mrs B (z=1). This qualifies for an instrument because, usually, the identity of the pollster can have an effect on the response but not on the value of the surveyed variable. If the missing data model is any from Section \ref{s2} and pollster B has a higher response rate than pollster A, then in the hypothetic situation where all individuals surveyed by Mr A had been surveyed by Mrs B, then those who responded to Mr A respond to Mrs B and some who did not respond to Mr A respond to Mrs B, but no one who responded to Mr A would not respond to Mrs B. This last type of individuals corresponds to the so-called defiers in the terminology of \cite{IA}: those for which $R_i=1$ when $z=1$ and  $R_i=0$ when $z=0$. There, instrument monotonicity means that there are no defiers. 
\begin{rmrk}
The terminology also calls compliers those who did not respond to Mr A but who would respond to Mrs B, never takers those who would respond to neither, and always takers those who would respond to both. 
\end{rmrk}
The absence of defiers can be unrealistic. For example, some surveyed individuals can answer a pollster because they feel confident with them. They can share the same traits which the statistician do not observe. For example, in the conversation they could realize they share the same interest or went to the same school.

\section{A Random Coefficients Model for the Selection Equation}\label{RC}
\cite{V} showed that monotonicity is equivalent to modelling the selection equation as an additively separable latent index model with a single unobservable. In \eqref{e2} the index is $\pi(\boldsymbol{Z})-H$ and $H$ is the unobservable. A nonadditively separable model takes the form  $\pi(\boldsymbol{Z},H)$. \cite{HV} proposes for a nonadditively separable index with multiple unobservables a random coefficients binary choice model. They call it a benchmark. A random coefficients latent index model takes the form $A+\boldsymbol{B}^{\top}\boldsymbol{Z}$, where $(A,\boldsymbol{B}^{\top})$ and $\boldsymbol{Z}$ are independent. This leads to
\begin{equation}\label{modRC}
R=1\{A+\boldsymbol{B}^{\top}\mathbf{Z}>0\}.
\end{equation}
 The multiple unobservables are the coefficients $(A,\boldsymbol{B}^{\top})$ and play the role of $H$ above. The model is nonadditively separable due to the products. The random intercept $A$ absorbs the usual mean zero error and deterministic intercept. The random slopes $\boldsymbol{B}$ can be interpreted as the taste for the characteristic $z$. The components of $(A,\boldsymbol{B}^{\top})$ can be dependent. 

To gain intuition, assume that $\boldsymbol{Z}$ is discrete. For $\bf{z}\in\mathrm{supp}(\boldsymbol{Z})$ and individuals $i\in\mathcal{I}(\bf{z})$ such that $\boldsymbol{Z}_i=\bf{z}$, we have 
$$R_i=1\{A_i+\boldsymbol{B}_i^{\top}\mathbf{z}>0\}.$$ Suppose that the first component of $\boldsymbol{B}$ takes  positive and negative values with positive probability, that we change exogeneously $\bf{z}$ to $\bf{z}'$ in $\mathrm{supp}(\boldsymbol{Z})$ by only changing the first component, and that we leave unchanged the unobserved characteristics $(A_i,\boldsymbol{B}^{\top}_i)$ for $i\in\mathcal{I}(\bf{z})$. This model allows for populations of 
compliers (those for which the first component of $\boldsymbol{B}_i$ is positive) and defiers (those for which the first component of $\boldsymbol{B}_i$ is negative).  

A parametric model for a selection equation specifies a parametric law for $(A,\boldsymbol{B}^{\top})$. A parametric model for a selection model specifies a joint law of $(A,\boldsymbol{B}^{\top},Y)$ given $\boldsymbol{X},\boldsymbol{Z}$. The model parameters can be estimated by maximum likelihood. The components of  $(A,\boldsymbol{B}^{\top},Y)$ given $\boldsymbol{X},\boldsymbol{Z}$ could be modelled as dependent. $(A,\boldsymbol{B}^{\top})$ is a vector of latent variables and the likelihood involves integrals over $\mathbb{R}^{d_{\boldsymbol{Z}}+1}$. As for the usual Logit or Probit models, a scale normalization is usually introduced for identification. Indeed $1\{A+\boldsymbol{B}^{\top} \boldsymbol{Z}>0\}=1\{c(A+\boldsymbol{B}^{\top} \boldsymbol{Z})>0\}$ for all $c>0$. A nonparametric model allows the law of 
$(A,\boldsymbol{B}^{\top},Y,\boldsymbol{Z})$ given $\boldsymbol{X}=\boldsymbol{x}$  to be a nonparametric class. Parametric and nonparametric models are particularly interesting when they allow for discrete mixtures to allow for different groups of individuals such as the compliers, defiers, always takers and never takers. But estimating a parametric model with latent variables which are drawn from multivariate mixtures can be a difficult exercise. In contrast nonparametric estimators can be easy to compute. 

The approach in this paper to relax monotonicity is based on \cite{GH}. 

\subsection{Scaling to Handle Genuine Non Instrument Monotonicity}
In this section, we rely on the approach used in  the first version of \cite{GH} in the context of treatment effects models. This is based on the normalisation in \cite{GK,GlP}. Let $d-1$ be the dimension  of the vector of instrumental variables. For scale normalization, we define
\begin{align*}
\boldsymbol{\boldsymbol{\Gamma}}^\top=\frac{\left(A,\boldsymbol{B}^\top\right)}{\left\|\left(A,\boldsymbol{B}^\top\right)\right\|}1\left\{\left(A,\boldsymbol{B}^\top\right)\ne0\right\},\quad\boldsymbol{S}^\top=\frac{\left(1,\boldsymbol{Z}^\top\right)}{\left\|\left(1,\boldsymbol{Z}^\top\right)\right\|}
\end{align*}
so that 
$$R=1\{\boldsymbol{\boldsymbol{\Gamma}}^\top\boldsymbol{S}>0\}.$$
We introduce some additional notations. When $f$ 
is an integrable function on $\mathbb{S}^{d-1}$, 
we denote by $\check{f}$ the function $\boldsymbol{\theta}\in \mathbb{S}^{d-1}\mapsto f(-\boldsymbol{\theta})$, by $f^-$ the function $(f-\check{f})/2$.
If $f\in L^2(\mathbb{S}^{d-1})$ (\emph{i.e.} is square integrable) is nonnegative a.e. and $f\check{f}=0$ a.e., then 
\begin{equation}\label{eInv0}
f=2f^-1\left\{f^->0\right\}\ \text{a.e.}
\end{equation}
The hemispherical transform (see \cite{R}) of an integrable function $f$ on $\mathbb{S}^{d-1}$ is defined as
$$\forall \boldsymbol{s}\in\mathbb{S}^{d-1},\ \mathcal{H}[f](\boldsymbol{s})=\int_{\boldsymbol{\theta}\in\mathbb{S}:\ \langle \boldsymbol{s},\boldsymbol{\theta} \rangle\ge 0}f(\boldsymbol{\theta})d\sigma(\boldsymbol{\theta}).$$ 
This is a circular convolution in dimension $d=2$
$$\forall \varphi\in[0,2\pi),\ \mathcal{H}[f](\varphi)=\int_{\varphi\in[0,2\pi):\ \cos(\varphi-\theta)\ge 0}f(\theta)d\theta.$$ 
We now recall a few useful properties of the hemispherical transform (see \cite{GK} for more details). If $f\in L^2(\mathbb{S}^{d-1})$, then $\mathcal{H}[f]$ is a continuous function and $\mathcal{H}[f^-]=\mathcal{H}[f]^-$. 
The null space of $\mathcal{H}$ consists of the integrable functions which are even (by a density argument) and integrate to 0 on $\mathbb{S}^{d-1}$. As a result 
\begin{equation}\label{eexplic}
\mathcal{H}[f]=\int_{\boldsymbol{\theta}\in\mathbb{S}^{d-1}}f(\boldsymbol{\theta})d\sigma(\boldsymbol{\theta})/2+\mathcal{H}[f^-].
\end{equation} 
$\mathcal{H}$ is injective when acting on the cone of nonnegative almost everywhere functions in $L^2(\mathbb{S})$  
such that $f\check{f}=0$ a.e. (see \cite{GK,GlP}). This means that $f$ cannot be nonzero at two antipodal points of $\mathbb{S}$. We denote by $\mathcal{H}^{-1}$ the unbounded inverse operator. 
We now present a formula for the inverse. If $g=\mathcal{H}(f)$, then  
\begin{align}
f^-(\boldsymbol{\gamma})&=
\sum_{p\in\mathbb{N}_0}
\frac{1}{\lambda_{2p+1,d}}
\int_{\mathbb{S}^{d-1}} q_{2p+1,d}(\boldsymbol{\gamma}^{\top}\boldsymbol{s})g(\boldsymbol{s})d\sigma(\boldsymbol{s})\label{eInv}\\
&=\sum_{p\in\mathbb{N}_0}
\frac{1}{\lambda_{2p+1,d}}
\int_{\mathbb{S}^{d-1}} q_{2p+1,d}(\boldsymbol{\gamma}^{\top}\boldsymbol{s})g^-(\boldsymbol{s})d\sigma(\boldsymbol{s}),
\label{eInv2}
\end{align}
where
\begin{align*}
&\lambda_{1,d}=\frac{|\mathbb{S}^{d-2}|} {d-1},\ \forall p \in \mathbb N,\
\lambda_{2p+1,d}=\frac{(-1)^{p}|\mathbb{S}^{d-2}|1\cdot3\cdots(2p-1)}
{(d-1)(d+1)\cdots(d+2p-1)},
\end{align*}
$$L(k,d)=\frac{(2k+d-2)(k+d-2)!}{k!(d-2)!(k+d-2)},\ q_{k,d}(t) :=
\frac{L(k,d)C_k^{(d-2)/2}(t)}{|\mathbb{S}^{d-1}|C_k^{(d-2)/2}(1)},$$
for all $\mu>-1/2$ and $k\in\mathbb{N}_0$, $C_k^{\mu}(t)$ are orthogonal polynomials on $[-1,1]$ for the weight $(1-t^2)^{\mu-1/2}dt$. The Gegenbauer polynomials $C_k^{\mu}(t)$ can be obtained by the recursion $C_0^{\mu}(t)=1$, $C_1^{\mu}(t)=2\mu t$
for $\mu\ne0$ while $C_1^{0}(t)=2t$, and
$$
(k+2)C_{k+2}^{\mu}(t)=2(\mu+k+1)tC_{k+1}^{\mu}(t)-(2\mu+k)C_k^{\mu}(t).
$$
Indeed, for all $p\in\mathbb{N}_0$,   $\boldsymbol{s}\to q_{2p+1,d}(\boldsymbol{\gamma}^{\top}\boldsymbol{s})$ is odd. 
\begin{rmrk}
Other inversion formulas when $\mathcal{H}$ is restricted to odd functions or measures rather than the above cone are given in \cite{R}. 
\end{rmrk}
We consider the following model restrictions, for all $\boldsymbol{x}\in\mathrm{supp}(\boldsymbol{X})$, 
\begin{align}
&\mathbb{P}\left(\boldsymbol{\boldsymbol{\Gamma}}=0|\boldsymbol{X}=\boldsymbol{x}\right)=0,\label{e7}\\
&\boldsymbol{S}\text{ is independent of }\left(\boldsymbol{\boldsymbol{\Gamma}}^\top,Y\right)\text{ given }\boldsymbol{X},\label{e9}\\
&\text{The conditional law of }\boldsymbol{\boldsymbol{\Gamma}}\text{ given }\boldsymbol{X}=\boldsymbol{x}\text{ is absolutely continuous }\notag\\
&\text{with respect to }\sigma\text{ and the density belongs to }
L^2(\mathbb{S}^{d-1}),\label{e12}\\
&\text{For a.e. }\boldsymbol{\boldsymbol{\gamma}}\in\mathbb{S}^{d-1},\ f_{\boldsymbol{\boldsymbol{\Gamma}}|\boldsymbol{X}=\boldsymbol{x}}(\boldsymbol{\boldsymbol{\gamma}})\check{f}_{\boldsymbol{\boldsymbol{\Gamma}}|\boldsymbol{X}=\boldsymbol{x}}(\boldsymbol{\boldsymbol{\gamma}})=0,\label{e11}\\
&\mathrm{supp}\left(\boldsymbol{S}|\boldsymbol{X}=\boldsymbol{x}\right)=\{\boldsymbol{s}\in\mathbb{S}^{d-1}:\ \boldsymbol{s}_1\ge0\}.\label{e10}
\end{align}
Let $g_{\phi,\boldsymbol{x}}
=\mathcal{H}\left[\mathbb{E}\left[\phi(Y)|\boldsymbol{X}=\boldsymbol{x},\boldsymbol{\boldsymbol{\Gamma}}=\cdot\right]f_{\boldsymbol{\Gamma}|\boldsymbol{X}=\boldsymbol{x}}(\cdot)\right]$, where, by a slight abuse of notations, the root $\mathbb{E}\left[\phi(Y)|\boldsymbol{X}=\boldsymbol{x},\boldsymbol{\boldsymbol{\Gamma}}=\cdot\right]f_{\boldsymbol{\Gamma}|\boldsymbol{X}=\boldsymbol{x}}(\cdot)$ is zero outside $\text{supp}(\boldsymbol{\Gamma}|\boldsymbol{X}=\boldsymbol{x})$. By the above properties of $\mathcal{H}$, we have 
$g_{\phi,\boldsymbol{x}}^-
=\mathcal{H}\left[\left(\mathbb{E}\left[\phi(Y)|\boldsymbol{X}=\boldsymbol{x},\boldsymbol{\boldsymbol{\Gamma}}=\cdot\right]f_{\boldsymbol{\Gamma}|\boldsymbol{X}=\boldsymbol{x}}(\cdot)\right)^-\right]
$. 

\eqref{e10} can be replaced by, for identifying classes $\Phi_{\boldsymbol{x}}$ of functions which are nonnegative a.e. and  a quasi-analytic classes $\mathcal{C}_{\boldsymbol{x}}$, 
\begin{equation}\label{e10b}
\left.
\begin{array}{ll}
\text{For all}\ \boldsymbol{x}\in\mathrm{supp}(\boldsymbol{X}),\\
\ \ \mathrm{supp}\left(\boldsymbol{S}|\boldsymbol{X}=\boldsymbol{x}\right)\text{ has a nonempty interior}\\
\ \ \text{For all }\phi\in\Phi_{\boldsymbol{x}},\
g_{\phi,\boldsymbol{x}}\in\mathcal{C}_{\boldsymbol{x}}.
\end{array}
\right\}
\end{equation}

This specification has the advantage that we do not assume that the researcher knows that one coefficient has a sign. Indeed it is easy to see that \eqref{e11} contains such an assumption as a subcase. It allows for non instrument monotonicity for all instruments.  Condition \eqref{e10} is demanding because it means that $\mathrm{supp}(\boldsymbol{Z}|\boldsymbol{X}=\boldsymbol{x})$ is the whole space for all $\boldsymbol{x}\in\mathrm{supp}(\boldsymbol{X})$. Hence we provide \eqref{e10b} which allows for an intermediate between nonparametric assumptions which are too demanding on the instruments and a parametric model. For further reference, we use the notation $H^+=\{\boldsymbol{s}\in\mathbb{S}^{d-1}:\ \boldsymbol{s}_1\ge0\}$. 
\begin{rmrk}
Proceeding like in \cite{GG,GG2} allows an index of the form $\pi(\boldsymbol{Z},\boldsymbol{H})$ where $\boldsymbol{Z}$ are instrumental variables and $\boldsymbol{H}$ is multidimensional of arbitrary dimension but has a sparse random series expansion on some classes of functions and the conditional law of $\boldsymbol{Z}$, given $\boldsymbol{X}=\boldsymbol{x}$, for all $\boldsymbol{x}\in\mathrm{supp}(\boldsymbol{X})$, can have a support which is a subspace of the whole space. This means that a nonparametric random coefficients linear index already captures a large class of nonadditively separable models with multiple unobservables. 
\end{rmrk}
Using successively \eqref{e9}, the law of iterated expectations 
and \eqref{e12}, and \eqref{eexplic}, 
we obtain that, for all $(\boldsymbol{s}^{\top},\boldsymbol{x}^{\top})\in\mathrm{supp}(\boldsymbol{S}^{\top},\boldsymbol{X}^{\top})$, 
\begin{align}
\mathbb{E}\left[\phi(Y)R|\boldsymbol{X}=\boldsymbol{x},\boldsymbol{S}=\boldsymbol{s}\right]&=
\mathbb{E}\left[\phi(Y)1\{\boldsymbol{\Gamma}^{\top}\boldsymbol{s}>0\}|\boldsymbol{X}=\boldsymbol{x}\right]
\label{rhs}\\
&=g_{\phi,\boldsymbol{x}}(\boldsymbol{s}),\label{e140}\\
&=\frac12\mathbb{E}\left[\phi(Y)|\boldsymbol{X}=\boldsymbol{x}\right]+g_{\phi,\boldsymbol{x}}^-(\boldsymbol{s}).\label{equ}
\end{align}
We obtain the following theorem which states that $\mathbb{E}\left[\phi(Y)|\boldsymbol{X}=\boldsymbol{x}\right]$ can be identified at infinity under \eqref{e10}. 
\begin{theorem}
\label{p1} Assume \eqref{e7}-\eqref{e10}. For all $\tilde{\boldsymbol{s}}$ on the boundary of $H^+$ and $\boldsymbol{x}\in\mathrm{supp}(\boldsymbol{X})$, 
\begin{equation*}
\mathbb{E}\left[\phi(Y)|\boldsymbol{X}=\boldsymbol{x}\right]= 
\lim_{\boldsymbol{s}\to \tilde{\boldsymbol{s}},\ \boldsymbol{s}\in H^+}\mathbb{E}\left[\phi(Y)R|\boldsymbol{X}=\boldsymbol{x},\boldsymbol{S}=\boldsymbol{s}\right]+\lim_{\boldsymbol{s}\to -\tilde{\boldsymbol{s}},\ \boldsymbol{s}\in H^+}\mathbb{E}\left[\phi(Y)R|\boldsymbol{X}=\boldsymbol{x},\boldsymbol{S}=\boldsymbol{s}\right].
\end{equation*}
\end{theorem}
\begin{proof} Let $\boldsymbol{x}\in\mathrm{supp}(\boldsymbol{X})$. 
The result follows from \eqref{equ} 
and the facts that $g_{\phi,\boldsymbol{x}}^-$ is odd and 
continuous as recalled at the beginning of the paragraph.  
\end{proof}
By \eqref{eInv0}, for all $\boldsymbol{x}\in\mathrm{supp}(\boldsymbol{X})$, the law of $(Y,\boldsymbol{\Gamma}^{\top})$ conditional on $\boldsymbol{X}=\boldsymbol{x}$ is identified if
\begin{equation}\label{rootident}
\left.
\begin{array}{l}
\text{For all }\boldsymbol{x}\in\mathrm{supp}(\boldsymbol{X}),\text{ there exists an identifying class }\Phi_{\boldsymbol{x}}\text{ of functions}\\
\text{which are nonnegative a.e. such that,}\\
\text{for all }\phi\in\Phi_{\boldsymbol{x}},\ 
(\mathbb{E}[\phi(Y)|\boldsymbol{X}=\boldsymbol{x},\boldsymbol{\boldsymbol{\Gamma}}=\cdot]f_{\boldsymbol{\Gamma}}(\cdot))^-\text{ is identified.}
\end{array}
\right\}
\end{equation}

%

The next theorem shows that, when $\phi$ is positive a.e., by integration, $\mathbb{E}\left[\phi(Y)|\boldsymbol{X}=\boldsymbol{x}\right]$ is nonparametrically identified with an alternative formula which does not involve taking limits. 
\begin{theorem}\label{p2} Assume \eqref{e7}-\eqref{e11} and either \eqref{e10} or \eqref{e10b}. \eqref{rootident} holds. 
Moreover, under \eqref{e10}, 
for all  $\boldsymbol{x}\in\mathrm{supp}(\boldsymbol{X})$, $\phi\in\Phi_{\boldsymbol{x}}$, $\boldsymbol{\boldsymbol{\gamma}}\in\mathbb{S}^{d-1}$, and $p\in\mathbb{N}_0$,
\begin{align}
&\int_{\mathbb{S}^{d-1}} q_{2p+1,d}(\boldsymbol{\gamma}^{\top}\boldsymbol{s})g_{\phi,\boldsymbol{x}}(\boldsymbol{s})d\sigma(\boldsymbol{s})\label{coef}\\&=2\mathbb{E}\left[\left.\frac{q_{2p+1,d}(\boldsymbol{\gamma}^{\top}\boldsymbol{S})}{f_{\boldsymbol{S}|\boldsymbol{X}=\boldsymbol{x}}(\boldsymbol{S})}\phi(Y)R\right|\boldsymbol{X}=\boldsymbol{x}\right]-\mathbb{E}\left[\phi(Y)|\boldsymbol{X}=\boldsymbol{x}\right]\int_{H^+} q_{2p+1,d}(\boldsymbol{\gamma}^{\top}\boldsymbol{s})(\boldsymbol{s})d\sigma(\boldsymbol{s}).\notag
\end{align}
\end{theorem}
\begin{proof} Let $\boldsymbol{x}\in\mathrm{supp}(\boldsymbol{X})$ and $\phi\in\Phi_{\boldsymbol{x}}$. Assuming \eqref{e10}, by Theorem \ref{p1}, \eqref{equ}, and the fact that $g_{\phi,\boldsymbol{x}}^-$ is odd, $g_{\phi,\boldsymbol{x}}^-$ is identified. Hence, \eqref{rootident} holds. 
By the right-hand side of \eqref{rhs}, 
$$g_{\phi,\boldsymbol{x}}(-\boldsymbol{s})=\mathbb{E}\left[\phi(Y)|\boldsymbol{X}=\boldsymbol{x}\right]-\mathbb{E}\left[\phi(Y)R|\boldsymbol{X}=\boldsymbol{x},\boldsymbol{S}=\boldsymbol{s}\right] ,$$
which yields, for all $p\in\mathbb{N}_0$,  
\begin{align*}
\int_{\mathbb{S}^{d-1}} q_{2p+1,d}(\boldsymbol{\gamma}^{\top}\boldsymbol{s})g_{\phi,\boldsymbol{x}}(\boldsymbol{s})d\sigma(\boldsymbol{s})=&2\int_{H^+} q_{2p+1,d}(\boldsymbol{\gamma}^{\top}\boldsymbol{s})\mathbb{E}\left[\phi(Y)R|\boldsymbol{X}=\boldsymbol{x},\boldsymbol{S}=\boldsymbol{s}\right]d\sigma(\boldsymbol{s})\\
&+\mathbb{E}\left[\phi(Y)|\boldsymbol{X}=\boldsymbol{x}\right]\int_{H^+} q_{2p+1,d}(-\boldsymbol{\gamma}^{\top}\boldsymbol{s})(\boldsymbol{s})d\sigma(\boldsymbol{s}),
\end{align*}
hence the moreover part.\\ 
Assuming \eqref{e10b}, by \eqref{e140} $g_{\phi,\boldsymbol{x}}$, hence $g_{\phi,\boldsymbol{x}}^-$, is identified. Hence, \eqref{rootident} holds. 
%
\end{proof}
\begin{rmrk}By taking  $\phi$ to be the function identically equal to 1,  
we obtain $f_{\boldsymbol{\boldsymbol{\Gamma}}|\boldsymbol{X}=\boldsymbol{x}}(\boldsymbol{\boldsymbol{\gamma}})$ for all $\boldsymbol{x}$ and a.e. $\boldsymbol{\gamma}$ such that $(\boldsymbol{x},\boldsymbol{\gamma})\in\mathrm{supp}(\boldsymbol{X},\boldsymbol{\Gamma})$. 
\end{rmrk}
A simple estimator of $(\mathbb{E}[\phi(Y)|\boldsymbol{X}=\boldsymbol{x},\boldsymbol{\boldsymbol{\Gamma}}=\cdot]f_{\boldsymbol{\Gamma}}(\cdot))^-$ on a grid of $\boldsymbol{\gamma}$ on $\mathbb{S}^{d-1}$ under \eqref{e10} takes the form
$$\sum_{p=1}^T
\frac{\hat{c}_{2p+1}(\boldsymbol{\gamma})}{\lambda_{2p+1,d}},$$
where $\hat{c}_{2p+1}(\boldsymbol{\gamma})$ are estimators of the integrals in \eqref{coef} and $T$ is a smoothing parameter. 
This yields $\mathbb{E}[\phi(Y)|\boldsymbol{X}=\boldsymbol{x},\boldsymbol{\boldsymbol{\Gamma}}=\cdot]f_{\boldsymbol{\Gamma}}(\cdot)$ using a plug-in and 
\eqref{eInv0}.  A useful choice of $\phi$ for Algorithm \ref{algmsim} is $1\{.\le t\}$ for $t$ on a grid on $\mathbb{R}$. \begin{algorithm}\label{algmsphere1}
$\hat{c}_{2p+1}(\boldsymbol{\gamma})$, for all $p=1,\dots,T$, are obtained as follows:
\begin{enumerate}
\item Compute $\mathbb{E}\left[\phi(Y)|\boldsymbol{X}=\boldsymbol{x}\right]$ using a local polynomial estimator of the right-hand side of the identity in Theorem \ref{p1} and compute numerically $\int_{H^+} q_{2p+1,d}(\boldsymbol{\gamma}^{\top}\boldsymbol{s})(\boldsymbol{s})d\sigma(\boldsymbol{s})$, 
\item Form, for the observations $i=1,\dots,N$ in the sample, 
$$\frac{q_{2p+1,d}(\boldsymbol{\gamma}^{\top}\boldsymbol{S}_i)}{\hat{f}_{\boldsymbol{S}|\boldsymbol{X}=\boldsymbol{x}}(\boldsymbol{S}_i)}\phi(Y_i)R_i,$$
where $\hat{f}_{\boldsymbol{S}|\boldsymbol{X}=\boldsymbol{x}}$ is a density estimator for directional data (see, \emph{eg}, \cite{GK}),  
and estimate $\mathbb{E}\left[\left.\frac{q_{2p+1,d}(\boldsymbol{\gamma}^{\top}\boldsymbol{S})}{f_{\boldsymbol{S}|\boldsymbol{X}=\boldsymbol{x}}(\boldsymbol{S})}\phi(Y)R\right|\boldsymbol{X}=\boldsymbol{x}\right]$ using a local polynomial estimator. 
\end{enumerate}
\end{algorithm}
In the approach in \cite{GK}, there is an additional damping of the high frequencies by an infinitely differentiable filter with compact support. The needlet estimator in \cite{GlP} also builds on this idea. In the case of the estimation of $f_{\boldsymbol{\boldsymbol{\Gamma}}|\boldsymbol{X}=\boldsymbol{x}}$, \cite{GlP} provides the minimax lower bounds for more general losses and an adaptive estimator based on thresholding the coefficients of a needlet expansion with a data driven level of hard thresholding. 

Building an estimator based on \eqref{e10b}, Hilbert space techniques, and assuming analyticity is an ongoing project. 

To perform Algorithm \ref{algmsim}, it is not 
useful to estimate the whole $\mathbb{E}[\phi(Y)|\boldsymbol{X}=\boldsymbol{x},\boldsymbol{\boldsymbol{\Gamma}}=\cdot]f_{\boldsymbol{\Gamma}}(\cdot)$. Rather, the estimator $\mathbb{E}\left[\phi(Y)|\boldsymbol{X}=\boldsymbol{x}\right]$ and local polynomial estimators are enough to obtain the elements in \eqref{e6b}.

\subsection{Alternative Scaling Under a Weak Version of Monotonicity}
In this section, we still assume \eqref{modRC} and $\boldsymbol{Z}$ is independent of 
$\left(A,\boldsymbol{B}^\top,Y\right)$ given $\boldsymbol{X}$ (\eqref{Px} under the previous normalization). 
We maintain as well 
\begin{equation}\label{Px}
\text{For all }\boldsymbol{x}\in\mathrm{supp}(\boldsymbol{X}),\ \exists \boldsymbol{P}_{\boldsymbol{x}}\in GL(d-1):\  \left(\boldsymbol{P}_{\boldsymbol{x}}^{\top}\boldsymbol{B}\right)_1> 0\ a.s.,
\end{equation}
where $GL(d-1)$ the general linear group over $\mathbb{R}^{d-1}$. 

Under this assumption we can rewrite the model as follows. We denote by $V=(\boldsymbol{P}_{\boldsymbol{x}}^{-1}\boldsymbol{Z})_1$, $\overline{\boldsymbol{Z}}=(\boldsymbol{P}_{\boldsymbol{x}}^{-1}\boldsymbol{Z})_{2,\dots,d-1}$, 
$\Theta=-A/\left(\boldsymbol{P}_{\boldsymbol{x}}^{\top}\boldsymbol{B}\right)_1$, and $\overline{\boldsymbol{\Gamma}}=-\left(\boldsymbol{P}_{\boldsymbol{x}}^{\top}\boldsymbol{B}\right)_{2,\dots,d-1}/\left(\boldsymbol{P}_{\boldsymbol{x}}^{\top}\boldsymbol{B}\right)_1$.  This yields
$$A+\boldsymbol{B}^{\top}\boldsymbol{Z}>0\Leftrightarrow V-\Theta-\overline{\boldsymbol{\Gamma}}^{\top}\overline{\boldsymbol{Z}}>0,$$
hence
\begin{equation}\label{modsr}
R=1\{V-\Theta-\overline{\boldsymbol{\Gamma}}^{\top}\overline{\boldsymbol{Z}}>0\}
\end{equation}
and 
\begin{equation}\label{modsr2}(V,\overline{\boldsymbol{Z}}^{\top}) \text{ is independent of }
(\Theta,\overline{\boldsymbol{\Gamma}}^\top,Y)\text{ given }\boldsymbol{X}.
\end{equation} 
By \eqref{modsr2}, \eqref{modsr} is equivalent to the fact that, for all $\boldsymbol{x}\in\mathrm{supp}(\boldsymbol{X})$ and $\overline{\boldsymbol{z}}\in\mathrm{supp}(\overline{\boldsymbol{Z}})$, $$v\to \mathbb{P}(R=1|\boldsymbol{X}=\boldsymbol{x},\boldsymbol{Z}=\boldsymbol{P}_{\boldsymbol{x}}(v,\ \overline{\boldsymbol{z}}^{\top})^{\top})=\mathbb{P}(\Theta+\overline{\boldsymbol{\Gamma}}^{\top}\overline{\boldsymbol{z}}<v|\boldsymbol{X}=\boldsymbol{x})$$
is a cumulative distribution, so the researcher can determine, from the distribution of the data, such an invertible matrix 
$\boldsymbol{P}_{\boldsymbol{x}}$.   

The  vector $(1\ -\Theta\ -\overline{\boldsymbol{\Gamma}}^{\top})^{\top}$ of random coefficients in the linear index $V-\Theta-\overline{\boldsymbol{\Gamma}}^{\top}\overline{\boldsymbol{Z}}$ clearly satisfies \eqref{e11}.
For this reason the specification of the previous section is more general. There is instrument monotonicity in $V$, though not for $\overline{\boldsymbol{Z}}$. This is a weak type of monotonicity because it is possible that there is instrument monotonicity for none of the instrumental variable in the original scale. This is the approach presented in the other versions of \cite{GH}. It is shown in \cite{GH} that the equation
$$
R=\indic\left\{V+f_0\left(\widetilde{\boldsymbol{Z}}\right)- \Theta  -\sum_{l=1}^{d-2}\overline{\boldsymbol{\Gamma}}_lf_l\left(\widetilde{\boldsymbol{Z}}_l\right)>0\right\},$$
where $d\ge3$, $f_0,\dots,\ f_{d-2}$ are unknown functions, can be transformed by reparametrization into \eqref{modsr} and the unknown functions are identified by similar arguments as for the additive model for a regression function.

We consider as well the following restrictions: 
\begin{equation}
\text{For all }(\boldsymbol{x}^{\top},\ \overline{\boldsymbol{z}}^{\top})\in\mathrm{supp}(\boldsymbol{X}^{\top},\ \overline{\boldsymbol{Z}}^{\top}),\ f_{\Theta,\overline{\boldsymbol{\Gamma}}|\boldsymbol{X}=\boldsymbol{x}}\text{ and }f_{\Theta+\overline{\boldsymbol{\Gamma}}^{\top}\overline{\boldsymbol{z}}|\boldsymbol{X}=\boldsymbol{x}}\text{ exist,}
\label{a00}
\end{equation}
and either
\begin{equation}
\left.
\begin{array}{ll}
&\forall (\boldsymbol{x},\overline{\boldsymbol{z}})\in\text{supp}(\boldsymbol{X},\overline{\boldsymbol{Z}}),\ 
\mathrm{supp}\left(V|\boldsymbol{X}=\boldsymbol{x},\overline{\boldsymbol{Z}}=\overline{\boldsymbol{z}}\right)\supseteq 
\mathrm{supp}(\Theta +\overline{\boldsymbol{\Gamma}}^{\top}\overline{\boldsymbol{z}}|\boldsymbol{X}=\boldsymbol{x}),\\
&\forall \boldsymbol{x}\in\text{supp}(\boldsymbol{X}),\ \text{supp}(\overline{\boldsymbol{Z}}|\boldsymbol{X}=\boldsymbol{x})=\mathbb{R}^{d-2},
\end{array}
\right\}\label{a0}
\end{equation}
or, for identifying classes $\Phi_{\boldsymbol{x}}$ and quasi-analytic classes $\mathcal{C}_{\boldsymbol{x},\overline{\boldsymbol{z}}}^a$ and $\mathcal{C}_{s,\boldsymbol{x}}^b$, denoting by
\begin{align*}
a_{\phi,\boldsymbol{x},\overline{\boldsymbol{z}}}&=\mathbb{E}[\phi(Y)|\Theta+\overline{\boldsymbol{\Gamma}}^{\top}\overline{\boldsymbol{z}}=\cdot,\boldsymbol{X}=\boldsymbol{x}]f_{\Theta+\overline{\boldsymbol{\Gamma}}^{\top}\overline{\boldsymbol{z}}|\boldsymbol{X}=\boldsymbol{x}}(\cdot),\\
b_{\phi,s,\boldsymbol{x}}&=\mathbb{E}[e^{i\Theta s}e^{i(\overline{\boldsymbol{\Gamma}}^{\top}\cdot)s}\phi(Y)|\boldsymbol{X}=\boldsymbol{x}],
\end{align*}
\begin{equation}\label{e10bb}
\left.
\begin{array}{ll}\text{For all}\ (\boldsymbol{x}^{\top},\overline{\boldsymbol{z}}^{\top})\in\text{supp}(\boldsymbol{X}^{\top},\overline{\boldsymbol{Z}}^{\top}),\ 
\mathrm{supp}\left(V|\boldsymbol{X}=\boldsymbol{x},\overline{\boldsymbol{Z}}=\overline{\boldsymbol{z}}\right)
\\
\text{and }\text{supp}(\overline{\boldsymbol{Z}}|\boldsymbol{X}=\boldsymbol{x})\text{ have nonempty interiors,}\\ 
\text{For all }\phi\in\Phi_{\boldsymbol{x}},\
a_{\phi,\boldsymbol{x},\overline{\boldsymbol{z}}}\in\mathcal{C}_{\boldsymbol{x},\overline{\boldsymbol{z}}}^a\text{ and for all }s\in\mathbb{R},\ b_{\phi,s,\boldsymbol{x}}\in\mathcal{C}_{s,\boldsymbol{x}}^b.
\end{array}
\right\}
\end{equation}
Clearly, for $a_{\phi,\boldsymbol{x},\overline{\boldsymbol{z}}}\in\mathcal{C}_{\boldsymbol{x}}$ to hold it is necessary that $f_{\Theta+\overline{\boldsymbol{\Gamma}}^{\top}\overline{\boldsymbol{z}}|\boldsymbol{X}=\boldsymbol{x}}$ is positive a.e.. A simple sufficient condition for $b_{\phi,s,\boldsymbol{x}}$ to be analytic is 
\begin{equation*}
\text{For all }\boldsymbol{x}\in\mathrm{supp}(\boldsymbol{X}),\ \exists R>0:\ \mathbb{E}\left[\left.\exp\left(R\left\|\overline{\boldsymbol{\Gamma}}\right\|\right)\right| \boldsymbol{X}=\boldsymbol{x}\right]<\infty,
\end{equation*}
This condition (which imply that $\overline{\boldsymbol{\Gamma}}$ does not have heavy tails) and the support conditions in \eqref{e10b} are slightly stronger than necessary (see \cite{GG}).


\begin{theorem}\label{p3}
Maintain \eqref{modsr}-\eqref{a00} and either \eqref{a0} or \eqref{e10bb}.  For all $\boldsymbol{x}\in\mathrm{supp}(\boldsymbol{X})$, the law of $(Y,\Theta,\overline{\boldsymbol{\Gamma}}^{\top})$ conditional on $\boldsymbol{X}=\boldsymbol{x}$ is identified.
\end{theorem}
\begin{proof}
Let $(\boldsymbol{x},\overline{\boldsymbol{z}})\in\text{supp}(\boldsymbol{X},\overline{\boldsymbol{Z}})$ and $\phi\in\Phi_{\boldsymbol{x}}$.  For all $v$ in the interior of $\mathrm{supp}(V|\boldsymbol{X}=\boldsymbol{x},\overline{\boldsymbol{Z}}=\overline{\boldsymbol{z}})$, we have
$$\partial_v\mathbb{E}\left[\left.\phi(Y)R\right|\boldsymbol{X}=\boldsymbol{x},V=v,\overline{\boldsymbol{Z}}=\overline{\boldsymbol{z}}\right]=\mathbb{E}\left[\left.\phi(Y)\right|\boldsymbol{X}=\boldsymbol{x},\Theta+\overline{\boldsymbol{\Gamma}}^{\top}\overline{\boldsymbol{z}}=v\right]f_{\Theta+\overline{\boldsymbol{\Gamma}}^{\top}\overline{\boldsymbol{z}}|\boldsymbol{X}=\boldsymbol{x}}(v).$$
So, by the assumptions, the above right-hand side is identified for all $v\in\mathbb{R}$. Hence, for all $s\in\mathbb{R}$, 
\begin{equation}\label{bphi}b_{\phi,s,\boldsymbol{x}}=\int_{\mathbb{R}}e^{isv}\mathbb{E}\left[\left.\phi(Y)\right|\boldsymbol{X}=\boldsymbol{x},\Theta+\overline{\boldsymbol{\Gamma}}^{\top}\overline{\boldsymbol{z}}=v\right]f_{\Theta+\overline{\boldsymbol{\Gamma}}^{\top}\overline{\boldsymbol{z}}|\boldsymbol{X}=\boldsymbol{x}}(v)dv
\end{equation}
is identified on $\text{supp}(\overline{\boldsymbol{Z}}|\boldsymbol{X}=\boldsymbol{x})$. We conclude using either the large support assumption or the now usual argument involving quasi-analyticity. 
\end{proof}
Based on \eqref{bphi}, it is not difficult to obtain an estimator of $b_{\phi,s,\boldsymbol{x}}$ under \eqref{a0} and then the root $\mathbb{E}[\phi(Y)|\boldsymbol{X}=\boldsymbol{x},\Theta,\overline{\boldsymbol{\Gamma}}=\cdot]f_{\Theta,\overline{\boldsymbol{\Gamma}}|\boldsymbol{X}=\boldsymbol{x}}(\cdot)$. 
\begin{algorithm}\label{algmnorm2}
\begin{enumerate}
\item Compute a local polynomial estimator of $\partial_v\mathbb{E}\left[\left.\phi(Y)R\right|\boldsymbol{X}=\boldsymbol{x},V=v,\overline{\boldsymbol{Z}}=\overline{\boldsymbol{z}}\right]$, 
\item Take a smooth numerical approximation of the Fourier transform of it, 
\item Use a smoothed multivariate inverse Fourier transform and a change of variable $(s,s\overline{\boldsymbol{z}})\to (s,\boldsymbol{z})$.
 \end{enumerate}
 \end{algorithm}
 Alternatively, \cite{GH} uses a smooth regularized inverse of the Radon transform and an integration by part. 
It is also possible to turn the identification argument based on \eqref{e10bb} 
into an estimation procedure as in \cite{GG3}. 

To perform Algorithm \ref{algmsim}, it is not 
useful to estimate the whole $\mathbb{E}[\phi(Y)|\boldsymbol{X}=\boldsymbol{x},\Theta,\overline{\boldsymbol{\Gamma}}=\cdot]f_{\Theta,\overline{\boldsymbol{\Gamma}}|\boldsymbol{X}=\boldsymbol{x}}(\cdot)$. Rather, one can estimate $\mathbb{E}\left[\phi(Y)|\boldsymbol{X}=\boldsymbol{x}\right]$ by steps 1 and 2  (for $s=0$) of Algorithm \ref{algmnorm2} and use local polynomial estimators of the remaining elements in \eqref{e6b}.

\begin{rmrk}\label{r1}
Proceeding like in \cite{GG,GG2} allows to work with an index of the form $\pi(\boldsymbol{Z},\boldsymbol{H})-V$ where $\boldsymbol{H}$ is multidimensional of arbitrary dimension and $\pi(\boldsymbol{Z},\boldsymbol{H})$ has a sparse random series expansion on some classes of functions and the conditional laws of $\boldsymbol{Z}$ and $V$, given $\boldsymbol{X}=\boldsymbol{x}$, for all $\boldsymbol{x}\in\mathrm{supp}(\boldsymbol{X})$, can have a support which is a subspace of the whole space. 
\end{rmrk}
\begin{rmrk}
In a binary treatment effect model the outcome can be written as $Y=(1-R)Y_0+RY_1$. $Y_0$ and $Y_1$ are the potential outcomes without and with treatment. They are unobservable. A selection model can be viewed as a degenerate case where $Y_0=0$ a.s.  
Quantities similar to the root in Theorem \ref{p2} have been introduced in \cite{GH}. They are for the marginals of the potential outcomes $\mathbb{E}\left[\phi(Y_j)|\boldsymbol{X}=\boldsymbol{x},\Theta=\theta,\overline{\boldsymbol{\Gamma}}=\overline{\boldsymbol{\gamma}}\right]$ for $j\in\{0,1\}$. An extension of  the Marginal Treatment Effect in \cite{HV} to multiple unobservables and for laws is the Conditional on Unobservables Distribution of Treatment Effects $\mathbb{E}\left[\phi(Y_1-Y_0)|\boldsymbol{X}=\boldsymbol{x},\Theta=\theta,\overline{\boldsymbol{\Gamma}}=\overline{\boldsymbol{\gamma}}\right]$. \cite{GH} considers kernel estimators which rely on regularized inverses of the Radon transform. 
\end{rmrk}

\section{Application to Missing Data in Surveys}\label{sec:1}
When making inference with survey data, the researcher has available data on a vector of characteristics for units belonging to a random subset $\mathcal{S}$ of a larger finite population $\mathcal{U}$. The law used to draw $\mathcal{S}$ can depend on variables available for the whole population, for example from a census.  We assume that the researcher is interested in a parameter $g$ which could be computed if we had the values of a variable $y_i$ for all units of index $i\in\mathcal{U}$. This can be an inequality index, for example the Gini index, and $y_i$ the wealth of household $i$. In the absence of missing data, the statistician can produce a confidence interval for $g$, making use of the data for the units $i\in\mathcal{S}$ and his available knowledge on the law $\mathcal{S}$. We assume that the cardinality of $\mathcal{S}$ is fixed and equal to $n$. When $g$ is a total, it is usual to rely on an unbiased estimator, an estimator of its variance, and a Gaussian approximation. 
For more complex parameters, linearization is often used to approximate moments. The estimator usually rely on the survey weights $\pi_i=1/\mathbb{P}(i\in\mathcal{S})$. For example an estimator of the Gini index is     
\begin{equation}\label{estg}
\widehat{g}\left((y_i)_{i\in\mathcal{S}}\right)=\frac{\sum_{i=1}^n(2\hat{r}(i)-1)\pi_iy_i}{\sum_{i=1}^n\pi_i\sum_{i=1}^n\pi_iy_i}-1,
\end{equation}
where $\hat{r}(i)=\sum_{j=1}^n w_j\indic\left\{y_j\le y_i\right\}$. The  estimators of the variance of the estimators are more complex to obtain and we assume there is a numerical procedure to obtain them. Inference is based on the approximation
\begin{equation}\label{norma}\widehat{g}\left((y_i)_{i\in\mathcal{S}}\right)\approx g+\sqrt{\widehat{var}\left(\widehat{g}\right)\left((y_i)_{i\in\mathcal{S}}\right)}\epsilon,
\end{equation}
where $\epsilon$ is a standard normal random variable and $\widehat{var}\left(\widehat{g}\right)\left((y_i)_{i\in\mathcal{S}}\right)$ is an estimator of the variance of $\widehat{g}\left((y_i)_{i\in\mathcal{S}}\right)$.

In practice, this is not possible when some of the $y_i$s are missing. There is a distinction between total nonresponse, where the researcher discards the data for some units $i\in\mathcal{S}$ or it is not available, and partial nonresponse. Let us ignore total nonresponse which is usually dealt with using reweighting and calibration and focus on partial nonresponse.
We consider a case where $y_i$ can be missing for some units $i\in\mathcal{S}$, while all other variables are available for all units $i\in\mathcal{S}$. 
We rely on a classical formalism where the vector of surveyed variables and of those used to draw $\mathcal{S}\subsetneq \mathcal{U}$, for each unit $i\in\mathcal{U}$, are random draws from a superpopulation. In this formalism the parameter $y_i$ for all indices $i$ of households in the population and $g$ are random and we shall now use capital letters for them.  
Let $S_i$ and $R_i$ be random variables, where $S_i=1$ if $i\in\mathcal{S}$ and $R_i=1$ if unit $i$ reveals the value of $Y_i$ given $S_i=1$, and $\boldsymbol{X}_i$ and $\boldsymbol{Z}_i$ be random vectors which will play a different role. 

It is classical to rely on imputations to handle the missing data. This means that we replace missing data by artificial values obtained from a model forming predictions or simulating from a probability law and inject them in a formula like \eqref{estg}. 
In \cite{G1} we discuss the use of the Heckman selection model when we suspect that the data is not missing at random. This relies on a parametric model for the partially missing outcome which is prone to criticism.  Also as this paper has shown such a model relies on instrument monotonicity which is an assumption which is too strong to be realistic. 

It is difficult to analyze theoretically the effect of such imputations. For example when the statistic is nonlinear in the $y_i$s (\emph{e.g.} \eqref{estg}) then using predictions can lead to distorted statistics. It is also tricky to make proper inference when one relies on imputations. One way to proceed is to rely on a hierarchical model as in \cite{G2}. There the imputation model is parametric and we adopted the Bayesian paradigm for two reasons.
The first is to account for parameter uncertainty and the second is to replace maximum likelihood with high dimensional integrals by a Monte Carlo Markov Chain Algorithm (a Gibbs sampler). The hierarchical approach also allows layers such as to model model uncertainty. The Markov chain produces sequences of values for each $Y_i$ for $i\in \mathcal{S}\setminus\mathcal{R}$ in the posterior distribution given $\left(\boldsymbol{W}_i\right)_{i\in\mathcal{S}}$, the choice of which is discussed afterwards. Subsequently we get a path of  
\begin{equation}\label{einvert}
\widetilde{G}=\widehat{G}\left(\left(Y_i\right)_{i\in\mathcal{S}}\right)+\sqrt{\widehat{var}\left(\widehat{G}\right)\left(\left(Y_i\right)_{i\in\mathcal{S}}\right)}\epsilon
\end{equation} 
where $\epsilon$ is a standard normal random variable independent from $\left(Y_i\right)_{i\in\mathcal{S}}$ given $\left(\boldsymbol{W}_i\right)_{i\in\mathcal{S}}$. \eqref{einvert} is derived from \eqref{norma}. The variables $\left(\boldsymbol{W}_i\right)_{i\in\mathcal{S}}$ are 
those making the missing mechanism corresponding to $R_i$ relative to $Y_i$ MAR\footnote{They can be those used by the survey statistician to draw $\mathcal{S}$ if any (and usually made available) to handle a total nonresponse which is MAR via imputations.}. The last $T$ values $\left(\widetilde{G}_t\right)_{t=T_0+1}^{T_0+T}$ of the sample path for $G$ allows to form credible sets $C$ by adjusting the set so that the frequency that $\left\{\widetilde{G}_t\in C\right\}$ exceeds $1-\alpha$, where $\alpha$ is a confidence level. $T_0$ is the so-called burn-in. These confidence sets account for error due to survey sampling, parameter uncertainty, and nonresponse. They can be chosen from the quantiles of the distribution, to minimize the volume of the set, etc. 

We now consider our nonparametric models of endogenous selection which allow for nonmonotonicity of the instrumental variables to handle a missing mechanism corresponding to $R$ which is NMAR. For simplicity, we assume away parameter uncertainty, which would be taken into account more easily if we adopted a Bayesian framework, and total nonresponse. The variables $\boldsymbol{X}_i$ in Section \ref{RC} can be variables that are good predictors for $Y_i$. They are not needed to obtain valid inference but can be useful to make confidence intervals smaller. 
However, the selection corresponding to the binary variables $R_i$ relative to the outcomes $Y_i$ given $S_i=1$ follow a NMAR mechanism. The (multiple) imputation approach becomes: for $t=1,\dots,T$
\begin{enumerate}
\item Draw an i.i.d. sample of $Y_i^t$ for $i\in\mathcal{S}\setminus\mathcal{R}$ from the law of $Y$ given $\boldsymbol{X}=\boldsymbol{x}_i$, $S=1$, and $R=0$, an independent standard normal $\epsilon_t$, and set $Y_i^t=y_i$ for $i\in\mathcal{R}$ where $y_i$ are the observations in the selected sample,
\item Compute 
\begin{equation}\label{einvertb}
\widetilde{G}_t=\widehat{G}\left(\left(Y_i^t\right)_{i\in\mathcal{S}}\right)+\sqrt{\widehat{var}\left(\widehat{G}\right)\left(\left(Y_i^t\right)_{i\in\mathcal{S}}\right)}\epsilon_t.
\end{equation}
\end{enumerate}
The confidence interval is formed from the sample $\left(\widetilde{G}_t\right)_{t=1,\dots,T}$ for a given confidence level. 

In practice, assuming away the conditioning on $\boldsymbol{X}$, the draws from the law of $Y$ given $S=1$, and $R=0$
can be obtained (approximately) by 
\begin{algorithm}\label{algmsim}
\begin{enumerate}
\item Take $\phi=1\{.\le t\}$ for a grid of $t$,
\item Estimate the left-hand side of \eqref{e6b} using plug-in estimators of the elements on the right-hand side from the available data (corresponding to $S=1$),
\item Draw from a uniform random variable on $[0,1]$,
\item Apply a numerical approximation of the inverse CDF from step 2.
\end{enumerate}
\end{algorithm}

\end{document}